\documentclass[a4paper]{amsart}

\usepackage{enumerate}
\usepackage{amssymb}
\usepackage{amsmath}
\allowdisplaybreaks[2]
\usepackage{mathrsfs}
\usepackage{amsthm}
\usepackage{mathtools}
\usepackage{tikz}
\usepackage[colorlinks,linkcolor=blue]{hyperref}
\usepackage{comment}
\theoremstyle{definition}
\newtheorem{theorem}{Theorem}[section]
\newtheorem{definition}[theorem]{Definition}

\newtheorem{lemma}[theorem]{Lemma}

\newtheorem{proposition}[theorem]{Proposition}
\newtheorem{corollary}[theorem]{Corollary}
\theoremstyle{remark}
\newtheorem{remark}[theorem]{Remark}

\numberwithin{equation}{section}

\begin{document}
	\title{A Weighted Llarull Type Theorem and its applications}
\author{Linfeng Zhou}
\address{Shanghai Key Laboratory of Pure Mathematics and Mathematical Practice, East
China Normal University, Shanghai}
\email{lfzhou@math.ecnu.edu.cn}
	\author{Guangrui Zhu}
    \address{School of Mathematical Sciences, East China Normal University, Shanghai, 200241, P. R. China}
\email{52275500052@stu.ecnu.edu.cn}
	\date{}
	\begin{abstract}
	This paper generalizes Llarull's classical scalar curvature rigidity theorem to the setting of weighted manifolds with P-scalar curvature. More precisely,  we prove the refinement of Llarull’s theorem for P-scalar curvature, which is similar to Listing's work \cite{listing2010scalar}. As an application, we establish a Llarull type theorem in the form of $\mathbb{S}^k\times\mathbb{T}^{n-k}$.
	\end{abstract}

	\maketitle
	\section{Introduction}
     The famous scalar curvature rigidity theorem of the round sphere was proved by Llarull in \cite{LlarullMR1600027}, which confirms the Gromov's conjecture in \cite{Gromov1}. For a comprehensive overview of extremality and rigidity properties of Riemannian manifolds, we refer to Gromov’s \textit{Four lectures on scalar curvature} \cite{Gromov2}. In fact, Llarull established the following theorem.
     \begin{theorem}(\cite[Theorem C]{LlarullMR1600027})\label{Llarull}
          Let $(M, g,)$ be a closed spin Riemannian manifold of dimension $n\geq3$ with $R(g)\geq n(n-1)$. If $h:(M,g) \to (\mathbb{S}^n, g_{\mathbb{S}^n})$ is a smooth, area non-increasing map of non-zero degree, then $h$ is an isometry.
     \end{theorem}
     This theorem reveals the interrelationship among metric, scalar curvature, and topology. As an improvement of Llarull’s theorem, Listing generalized previous results by omitting the assumption of $\|\wedge^2 h_*\|\leq1$ and using the function $\|\wedge^2 h_*\|$ in the scalar curvature inequality as follows.\par
     \begin{theorem}(\cite[Theorem 1 ]{listing2010scalar})\label{Listing}
         Let $(M^n, g)$ be a closed spin Riemannian manifold of dimension $n\geq3$. If $h:(M,g) \to (\mathbb{S}^n, g_{\mathbb{S}^n})$ is a smooth map of non-zero degree such that $R(g)\geq n(n-1)\|\wedge^2 h_*\|$, then there exists a constant $c>0$ such that $h:(M,c\cdot g) \to (\mathbb{S}^n, g_{\mathbb{S}^n})$ is an isometry.
     \end{theorem}
     \noindent Here $\|\wedge^2 h_*\|$ denotes the operator norm of the second exterior power of the differential of $h$, see Definition \ref{Def-h}. For the proof of this theorem, we also refer to Theorem 3.1 in \cite[Theorem 3.1]{W-XMR4853305}.\par
     The proofs of Theorem \ref{Llarull} and Theorem \ref{Listing} all rely on the Dirac operator method, requiring the hypothesis that M is spin. An open question is whether the spin assumption can be removed in both Theorem \ref{Llarull} and Theorem \ref{Listing}. Recently, Cecchini-Wang-Xie-Zhu \cite{cecchini2024scalar} proved Theorem \ref{Llarull} without the spin assumption in 4-dimension provided $h$ is a distance non-increasing map, Theorem \ref{Listing} plays an important role in their proof.\par
     
    A weighted Riemannian manifold is a Riemannian manifold $(M,g)$ equipped with a smooth function $f:M\to \mathbb{R}$ and a measure $e^{-f}dV_g$ on $M$. It was first introduced by Lichnerowicz in \cite{L1MR268812,L2MR300228}. Weighted manifolds have appeared for a long time in differential geometry. They play an important role in Perelman’s approach to the Ricci flow. A crucial idea of Perelman’s spectacular proofs \cite{perelman2002entropy} required considering manifolds with density and their evolution. This led him to introduce the notion of weighted scalar curvature
    \[R_{f}(g)\coloneqq R(g)+2\Delta_{g} f-|\nabla_g f|^2_g\,,\]
    which is not the trace of the weighted Ricci curvature of Bakry-\'Emery, sometimes called the P-scalar curvature. Weighted manifolds with positive or non-negative Weighted scalar curvature have very similar properties compared to Riemannian manifolds with positive or non-negative scalar
curvature. By the Gauss-Bonnet theorem and the divergence theorem, every oriented
closed surface with positive Weighted scalar curvature must be a topological sphere. In recent years, there have been numerous studies on weighted scalar curvature.  Schoen-Yau [31] showed that if $(M^3, g)$ is oriented with positive scalar curvature, then it contains no closed, immersed stable minimal surfaces of positive genus, Fan \cite{FEMR2407091} generalized this result to the weighted case (see also\cite{G-K-RMR3176286,MR2672304} for subsequent work). Baldauf-Ozuch \cite{BaldaufMR4470247} proved the positive weighted mass
theorem under the spin assumption by generalizing Witten’s spinor argument in \cite{Witten} to the weighted case. Chu-Zhu \cite{Chu-ZhuMR4765850} proved the positive weighted mass theorem under the dimension $3\leq n\leq7$ assumption by generalizing Schoen-Yau's minimal surface method in \cite{S-Y1526976,S-Y2MR612249} to the weighted case. For more works of the weighted scalar curvature supporting the above philosophy, we refer the reader to \cite{yMR4747071,yMR4662656,yMR4833746,L-MR4713241}. \par
    To state our main theorem, we need to introduce some definitions.
    \begin{definition}\label{Def-h}
        Suppose $h:(M,g)\to (N,\bar{g})$ is a smooth map between Riemannian manifolds $(M,g)$ and $(N,\bar{g})$.
        \begin{itemize}
            \item[(i)] 
        The operator norm of the differential of $h$ is defined as
        $$\|h_*\|:M\to[0,+\infty),\,p\,\to\max_{0\neq v\in T_pM}\left[\frac{h^*\bar{g}(v,v)}{g(v,v)}\right]^{\frac{1}{2}}.$$
        \item[(ii)]The operator norm of the second exterior power of the differential of $h$ is defined as
        $$\|\wedge^2 h_*\|:M\to[0,+\infty),\,p\,\to\max_{0\neq \eta\in \wedge^2 T_pM}\left[\frac{h^*\bar{g}(\eta,\eta)}{g(\eta,\eta)}\right]^{\frac{1}{2}}.$$
        \end{itemize}
       Then $h$ is said to be \textit{distance non-increasing} if $\|h_*\|\leq1$ and $h$ is said to be \textit{area non-increasing} if $\|\wedge^2 h_*\|\leq1$.
    \end{definition}
    \begin{remark}
        The area scaling can be bounded by the square of the distance scaling:
        \[0\leq\|\wedge^2 h_*\|\leq\|h_*\|^2.\]
    \end{remark}
    Our main results generalize Theorem \ref{Llarull} and Theorem \ref{Listing} to the weighted case.

    \begin{theorem}\label{main-theorem-2}
        Let $(M^n, g, e^{-f}\mathrm{d}V_g)$ be a weighted closed spin Riemannian manifold of dimension $n\geq3$. If $h:(M,g) \to (\mathbb{S}^n, g_{\mathbb{S}^n})$ is a smooth map of non-zero degree such that $R_{f}\geq n(n-1)\|\wedge^2 h_*\|$, then there exists a constant $c>0$ such that $h:(M,c\cdot g) \to (\mathbb{S}^n, g_{\mathbb{S}^n})$ is an isometry. Consequently, $f$ is a constant.
    \end{theorem}
   \begin{remark}
   The proof of Theorem \ref{main-theorem-2} follows Llarull's approach and using the Poincar\'e type inequality in \cite{W-XMR4853305}. The key step of our proof is to solve the weighted function $f$ and the operator norm $\|\wedge^2 h_*\|$ by constructing a perfect square expression and using integration by parts in a differential equation. 
    \end{remark}
    \begin{corollary}(\cite[Theorem 1.3]{DengMR4210894})\label{main-theorem-1}
    Let $(M, g, e^{-f}\mathrm{d}V_g)$ be a weighted closed spin Riemannian manifold of dimension $n\geq3$ with $R_{f}\geq n(n-1)$. If $h:(M,g) \to (\mathbb{S}^n, g_{\mathbb{S}^n})$ is a smooth, area non-increasing map of non-zero degree, then $h$ is an isometry. Consequently, $f$ is a constant.
    \end{corollary}
    
    \begin{remark}
       Deng has already proved Corollary \ref{main-theorem-1} for general weighted scalar curvature in \cite{DengMR4210894}.  
    \end{remark}
    As a corollary of Theorem \ref{main-theorem-2}, we obtain the 
 Llarull type theorem also holds for \textit{warped $\mathbb{T}^k$-extension} of $M$. see also results of B{\"a}r-Brendle-Hanke-Wang in \cite{BrendleMR4733718} and Wang-Xie in \cite{W-XMR4853305} for the case with a non-empty boundary.
 \begin{theorem}\label{T-extension}
     Let $(M^n,g)$ be a spin Riemannian manifold of dimension $n\geq3$, the \textit{warped $\mathbb{T}^k$-extension} of $M$ is the manifold  $(M^n\times\mathbb{T}^k,\bar{g})$ with $\bar{g}=g+\sum_{i=1}^ke^{-2f_i}d\theta^2$, where $f_i$ is a smooth function on $M$ for each $1\leq i\leq k$. Assume that $h:(M,g) \to (\mathbb{S}^n, g_{\mathbb{S}^n})$ is a smooth map of non-zero degree such that $R(\bar{g})\geq n(n-1)\|\wedge^2 h_*\|$, then there exists a constant $c>0$ such that $h:(M,c\cdot g) \to (\mathbb{S}^n, g_{\mathbb{S}^n})$ is an isometry. Consequently, $f_i$ is a constant for each $1\leq i\leq k$.
 \end{theorem}
 \begin{remark}
    For $n\geq2$, Theorem \ref{main-theorem-2} and Theorem \ref{T-extension} also hold if one replaces the operator norm $\|\wedge^2 h_*\|$ by $\|h_*\|^2$. 
\end{remark}
  As an application of Theorem \ref{T-extension}, we generalize Theorem 1.1 in \cite{HSSMR4855340} to $\mathbb{S}^k$ for $4\leq k\leq6$ by following the approach of Hao-Shi-Sun in \cite{HSSMR4855340}.
  \begin{theorem}\label{HSS}
    Suppose $4\leq n\leq7$ and $3\leq k\leq n-1$. Let $(M^n, g)$ be a closed spin Riemannian manifold. If $h:(M,g) \to (\mathbb{S}^k\times\mathbb{T}^{n-k}, g_{\mathbb{S}^k}+g_{\mathbb{T}^{n-k}})$ is a smooth map of non-zero degree such that $R(g)\geq k(k-1)\|\wedge^2 h_*\|$, then there exists a constant $c>0$ such that $(M,g)$ is locally isometric to $(\mathbb{S}^k\times\mathbb{T}^{n-k}, \frac{1}{c}\cdot(g_{\mathbb{S}^k}+g_{\mathbb{T}^{n-k}}))$.
      
  \end{theorem}
   Here, we would like to make a statement: The Theorem \ref{HSS} in the case of $h$ is area non-increasing was also obtained independently by Chow in a recent preprint \cite{chow2025scalar}.\par
    The paper is organized as follows. In section \ref{section-pre}, we give some facts about the spin manifolds and weighted Dirac operator. In section \ref{section-proof}, we prove Theorem \ref{main-theorem-2}. In section \ref{sec-HSS}, we prove Theorem \ref{T-extension} and Theorem \ref{HSS}.
    
    {\em Acknowledgements}: The second author would like to thank Professor Yukai Sun for much help and many discussions. 
    \section{Spin manifold and weighted Dirac operator}\label{section-pre}
    In this section, we recall some key facts concerning spin manifolds and weighted Dirac operator.\par 
    A spin manifold $M^n$ is an oriented manifold with a spin structure on its tangent bundle $TM$. Note that a manifold $M$ is spin if and only if the first and second Stiefel-Whitney class of tangent bundle vanish, namely $w_1(TM)=w_2(TM)=0$. There is a complex spinor bundle $S\coloneqq S(TM)$ with its Hermitian metric $\langle,\rangle_S$ and compatible connection $\nabla^S$ over a spin Riemannian manifold $(M,g)$. The Dirac operator is a first order elliptic differential operator acting on spinors and is given by
\[D\coloneqq \sum_{i=1}^{n} e_{i}\cdot\nabla^S_{e_i}\,,\]
   where $\{e_i\}_{i=1}^n$ is a local orthonormal frame of $TM$ and $\cdot$ denotes the Clifford multiplication. For a twisted spinor bundle $S\otimes E$, the twisted Dirac operator $D_E: \Gamma(S\otimes E)\to\Gamma(S\otimes E)$ is defined as
   \[D_E\coloneqq \sum_{i=1}^{n} e_{i}\cdot\nabla^{S\otimes E}_{e_i}\,,\]
   where $E$ is a given vector bundle over $M$.\par
    Let $(M,g)$ be a spin Riemannian manifold and $f \in C^{\infty}(M)$. The weighted Dirac operator $D_{f,E} : \Gamma(S\otimes E)\to\Gamma(S\otimes E)$ is defined as
\[ D_{f,E}\coloneqq D_E-\frac{1}{2}\nabla f \cdot ,\]
where $D_E$ is the standard twisted Dirac operator defined above, $\cdot$ denotes the Clifford multiplication. For instance, let $\varphi\otimes v\in\Gamma(S\otimes E)$, then
$$D_{f,E}(\varphi\otimes v)=D_E(\varphi\otimes v)-\frac{1}{2}(\nabla f\cdot\varphi)\otimes v\,.$$
\par
As standard Dirac operator, the weighted Dirac operator is also self-adjoint with respect to the weighted $L^2$-inner product. By slight abuse of notation, we denote the connection on $S\otimes E$ by $\nabla$.
\begin{proposition}\label{Dirac}
    The weighted Dirac operator $D_{f,E}$ is self-adjoint, i.e.,
    $$\int_{M}\langle D_{f,E} \Psi,\Phi\rangle e^{-f}=\int_{M}\langle  \Psi,D_{f,E}\Phi\rangle e^{-f},$$
    where $\Psi,\Phi\in\Gamma(S\otimes E)$ and $\langle,\rangle$ is the induce metic on $S\otimes E$ .
\end{proposition}
\begin{proof}
We fix an arbitrary point $p \in M$ and work with a local orthonormal
frame $\{e_i\}$ such that $\nabla e_i = 0$ at $p$. Then, at $p$, for any $\Psi,\Phi\in\Gamma(S\otimes E),$ 
\begin{align*}
    \langle D_{f,E}\Psi,\Phi\rangle&=\langle D_E\Psi-\frac{1}{2}\nabla f\cdot\Psi,\Phi\rangle\\
    &=\langle e_{i}\cdot\nabla_{e_i}\Psi,\Phi\rangle-\frac{1}{2}\langle\nabla f\cdot\Psi,\Phi\rangle\\
    &=-\langle\nabla_{e_i}\Psi,e_{i}\cdot\Phi\rangle+\frac{1}{2}\langle\Psi,\nabla f\cdot\Phi\rangle\\
    &=-e_i\langle\Psi,e_{i}\cdot\Phi\rangle+\langle\Psi,e_{i}\cdot\nabla_{e_i}\Phi\rangle+\frac{1}{2}\langle\Psi,\nabla f\cdot\Phi\rangle\\
    &=\langle\Psi,e_{i}\cdot\nabla_{e_i}\Phi\rangle-\frac{1}{2}\langle\Psi,\nabla f\cdot\Phi\rangle-e_i\langle\Psi,e_{i}\cdot\Phi\rangle+\langle\Psi,\nabla f\cdot\Phi\rangle\\
    &=\langle  \Psi,D_{f,E}\Phi\rangle-e_i\langle\Psi,e_{i}\cdot\Phi\rangle+\langle\Psi,\nabla f\cdot\Phi\rangle.\\
\end{align*}
The last two terms are a weighted divergence term. In fact, define a vector field
$X\in\Gamma(TM\otimes\mathbb{C})$ on $M$ by
\[g(X,Y)\coloneqq-\langle\Psi,Y\cdot\Phi\rangle\]
for all vector field $Y\in\Gamma(TM)$. Then (again computing at a fixed arbitrary point p)
\begin{align*}
\mathrm{div}_f X&\coloneqq\mathrm{div}X-g(X,\nabla f)\\
&=g(\nabla_{e_i}X,e_i)-g(X,\nabla f)\\
&=e_{i}g(X,e_i)-g(X,\nabla f)\\
&=-e_i\langle\Psi,e_{i}\cdot\Phi\rangle+\langle\Psi,\nabla f\cdot\Phi\rangle.
\end{align*}
The proposition follows from integration by parts.
\end{proof}
Similarly, the weighted Laplacian $\Delta_f: \Gamma(S\otimes E)\to \Gamma(S\otimes E)$ is defined as 
\[\Delta_f\coloneqq-\nabla^*\nabla-\nabla_{\nabla f},\]
it is also self-adjoint with respect to the weighted $L^2$-inner product.
\begin{proposition}\label{Laplacian}
    The weighted Laplacian $\Delta_f$ is a non-positive operator and is self-adjoint. In fact,
    \[\int_{M}\langle\Delta_f \Psi,\Phi\rangle e^{-f}=-\int_{M}\langle  \nabla\Psi,\nabla\Phi\rangle e^{-f}=\int_{M}\langle \Psi,\Delta_f\Phi\rangle e^{-f}.\]
\end{proposition}
\begin{proof}
Once again, we fix an arbitrary point $p \in M$ and work with a local orthonormal frame $\{e_i\}$ such that $\nabla e_i = 0$ at $p$. Then, at $p$, 
\begin{align*}
\langle\Delta_f\Psi,\Phi\rangle&=\langle\nabla_{e_i}\nabla_{e_i}\Psi,\Phi\rangle-\langle\nabla_{\nabla_f}\Psi,\Phi\rangle\\
    &=-\langle\nabla_{e_i}\Psi,\nabla_{e_i}\Phi\rangle+e_i\langle\nabla_{e_i}\Psi,\Phi\rangle-\langle\nabla_{\nabla_f}\Psi,\Phi\rangle\\
    &=-\langle\nabla\Psi,\nabla\Phi\rangle+e_i\langle\nabla_{e_i}\Psi,\Phi\rangle-\langle\nabla_{\nabla_f}\Psi,\Phi\rangle.
\end{align*}\par
As before, the last two terms are a weighted divergence term. In fact, define a vector field $U\in\Gamma(TM\otimes\mathbb{C})$ on $M$ by
\[g(U,V)\coloneqq\langle\nabla_{V}\Psi,\Phi\rangle\]
for all vector field $V\in\Gamma(TM)$. Then
\[\mathrm{div}_fU=e_i\langle\nabla_{e_i}\Psi,\Phi\rangle-\langle\nabla_{\nabla_f}\Psi,\Phi\rangle.\]
The desired formula follows by integration by parts.
\end{proof}
Similar to the standard Lichnerowicz formula, a weighted Lichnerowicz
formula holds, which was first observed by Perelman \cite{perelman2002entropy}. Now, we state and prove it here.
\begin{lemma}[weighted Lichnerowicz formula] Let $(M^n, g, e^{-f}\mathrm{d}V_g)$ be a weighted closed spin Riemannian manifold, then
    \[D_{f,E}^{2}=-\Delta_{f}+\frac{1}{4}R_{f}+\mathcal{R}^E,\] where $\mathcal{R}^E$ is defined on simple elements $\sigma\otimes v\in\Gamma(S\otimes E)$ by $\mathcal{R}^E(\sigma\otimes v)=\frac{1}{2}\Sigma_{i,j=1}^{n}\left(e_i\cdot e_j\cdot\sigma\right)\otimes R^E(e_i,e_j)v$, $R^{E}$ is the curvature tensor of $E$.
\end{lemma} 
\begin{proof}
    For any $\Phi\in\Gamma(S\otimes E)$, we have
    \begin{align*}
    D^2_{f,E}\Phi=&\left(D_E-\frac{1}{2}\nabla f\cdot \right)\left(D_E-\frac{1}{2}\nabla f\cdot \right)\Phi\\
    =&D_E^2\Phi-\frac{1}{2}\nabla f\cdot\left(e_i\cdot\nabla_{e_i}\Phi\right)-\frac{1}{2}\left(e_i\cdot\nabla_{e_i}\right)\left(\nabla f\cdot\Phi\right)-\frac{1}{4}\|\nabla f\|^2\Phi\\
    =&D_E^2\Phi-\frac{1}{2}\nabla f\cdot e_i\cdot\nabla_{e_i}\Phi-\frac{1}{2}e_i\cdot\langle\nabla_{e_i}\nabla f,e_j\rangle e_j\cdot\Phi\\&-\frac{1}{2}e_i\cdot\nabla f\cdot\nabla_{e_i}\Phi-\frac{1}{4}\|\nabla f\|^2\Phi\\
    =&D_E^2\Phi-\frac{1}{2}\left(e_i\cdot\nabla f+\nabla f\cdot e_i\right)\cdot\nabla_{e_i}\Phi-\frac{1}{2}\mathrm{Hess}f(e_i,e_j)e_i\cdot e_j\cdot\Phi\\
    &-\frac{1}{4}\|\nabla f\|^2\Phi\\
    =&\left(\nabla^*\nabla+\frac{1}{4}R+\mathcal{R}^E\right)\Phi+\langle\nabla f,e_i\rangle\nabla_{e_i}\Phi+\frac{1}{2}\left(\Delta f\right)\Phi-\frac{1}{4}\|\nabla f\|^2\Phi\\
    =&\left(\nabla^*\nabla+\nabla_{\nabla f}\right)\Phi+\frac{1}{4}\left(R+2\Delta f-\|\nabla f\|^2\right)\Phi+\mathcal{R}^E\Phi\\
    =&\left(-\Delta_{f}+\frac{1}{4}R_{f}+\mathcal{R}^E\right)\Phi.
     \end{align*}\par
    We have used the standard Lichnerowicz formula and the properties of Clifford multiplication.
\end{proof}
\section{Proof of Theorem \ref{main-theorem-2} }\label{section-proof}
In this section, We adopt Llarull’s setting to prove Theorem \ref{main-theorem-2} in two cases: even-dimensional case and odd-dimensional case as described in \cite{LlarullMR1600027}.
\begin{proof}[Proof of Theorem\ref{main-theorem-2}]
Let $\{e_i\}_{i=1}^{n}$ be a $g$-orthonormal frame near $p \in M$ and $\{\epsilon_i\}_{i=1}^{n}$ be a $g_{\mathbb{S}^{n}}$-orthonormal frame near $h(p) \in \mathbb{S}^{n}$. Moreover, we can require that there exist
non-negative scalars $\{\lambda_i\}_{i=1}^{n}$ such that $h_*(e_i)=\lambda_i\epsilon_i$. We will show that 
\begin{align}
    R_f&=n(n-1)\|\wedge^2 h_*\|
\end{align}
and
\begin{align} \lambda_i&=\|\wedge^2 h_*\|^{\frac{1}{2}}
\end{align}
for $1\leq i\leq n$.\par
\textit{Even Dimensional Case $n=2m\geq4$}:
Consider two complex spinor bundles $S$ over $M$ and $E_0$ over $\mathbb{S}^{2m}$, respectively. And consider the bundle $S\otimes E$ over $M$ for $E=h^*E_0$.
Let$\{\sigma_{\alpha}\}_{\alpha=1}^{2^{2m}}$ be a basis of $S$ and $\{v_{\beta}\}_{\beta=1}^{2^{2m}}$ be a basis of $E$, then for any $\Phi\in\Gamma(S\otimes E)$,
\[\Phi=a_{\alpha\beta}\sigma_\alpha\otimes v_{\beta}.\]
By Proposition \ref{Laplacian} and weighted Lichnerowicz formula, 
we have 
\begin{align*}
\int_{M}\langle D_{f,E}^2\Phi,\Phi\rangle e^{-f}=&\int_{M}\langle  \nabla\Phi,\nabla\Phi\rangle e^{-f}+\frac{1}{4}\int_{M}R_{f}\|\Phi\|^2e^{-f}+\int_{M}\langle\mathcal{R}^E\Phi,\Phi\rangle e^{-f}.\\
  \end{align*}  
According to Lemma 4.5 in \cite{LlarullMR1600027}, the curvature term has the following estimate,
\[\langle\mathcal{R}^E\Phi,\Phi\rangle\geq-\frac{1}{4}\sum_{i\neq j=1}^{2m}\lambda_i\lambda_j\|\Phi\|^2.\]
Consequently,
\begin{eqnarray*}
\int_{M}\langle D_{f,E}^2\Phi,\Phi\rangle e^{-f}\geq\int_{M}\|\nabla\Phi\|^2 e^{-f}+\frac{1}{4}\int_{M}(R_{f}-\sum_{i\neq j=1}^{2m}\lambda_i\lambda_j)\|\Phi\|^2e^{-f}.\\
  \end{eqnarray*}
  \par
Next, we prove $R_{f}=2m(2m-1)\|\wedge^2 h_*\|$ by contradiction. If $R_f(x)>2m(2m-1)\|\wedge^2 h_*\|(x)$ for some $x\in M$, then $ker (D_{f,E}^2)=0$, by Proposition \ref{Dirac}, $ker (D_{f,E})=0$. The weighted Dirac operator
$D_{f,E}: \Gamma((S\otimes E^+)\oplus(S\otimes E^-))\to \Gamma((S\otimes E^+)\oplus(S\otimes E^-))$ preserves the direct sum, so $D_{f,E^+}\coloneqq D_{f,E}|_{S\otimes E^+}$ has zero kernel, i.e. $ker(D_{f,E^+})=0$. We also have that
$D_{f,E^+}=D^+_{f,E^+}\oplus D^-_{f,E^+}$, where $D^{\pm}_{f,E^+}: \Gamma(S^{\pm}\otimes E^+)\to \Gamma(S^{\mp}\otimes E^+)$. Since $0=ker(D_{f,E^+})=ker(D^+_{f,E^+})\oplus ker(D^-_{f,E^+})$, we have $ker(D^+_{f,E^+})=ker(D^-_{f,E^+})=0$. Therefore the index of $D^+_{f,E^+}$ is given by
\[ Index(D^+_{f,E^+})=ker(D^+_{f,E^+})-ker(D^-_{f,E^+})=0.\]
We can define an elliptic operator $D^+_{f,E^+}(t)\coloneqq D^+_{E^+}-\frac{t}{2}\nabla f \cdot$ for $0\leq t\leq1$. It is easy to see that the weighted Dirac operator $D^+_{f,E^+}$ is homotopic to the standard Dirac operator $D^+_{E^+}$, where $D^+_{E^+}:\Gamma(S^{\pm}\otimes E^+)\to \Gamma(S^{\mp}\otimes E^+)$. Since the index is homotopy invariant, we can deduce that $Index(D^+_{E^+})=0$. But Atiyah-Singer index theorem states that $Index(D^+_{E^+})\neq0$ (it is equals to Euler characteristic of $\mathbb{S}^{2m}$ times degree of $h$), which leads to a contradiction. Therefore, $R_{f}\equiv2m(2m-1)\|\wedge^2 h_*\|$ on $M$. As $Index(D^+_{f,E^+})\neq0,\, ker(D_{f,E})\neq0$, there exists $0\neq\Psi\in\Gamma(S\otimes E)$ such that $D_{f,E}\Psi=0$ and
$$0\geq\int_{M}\| \nabla\Psi\|^2 e^{-f}+\frac{1}{4}\int_{M}\left(2m(2m-1)\|\wedge^2 h_*\|-\sum_{i\neq j=1}^{2m}\lambda_i\lambda_j\right)\|\Psi\|^2e^{-f},$$
combine with $\lambda_{i}\lambda_{j}\leq\|\wedge^2 h_*\|$ for $1\leq i\neq j\leq2m$, we have $\lambda_{i}\lambda_{j}=\|\wedge^2 h_*\|$ for $1\leq i\neq j\leq2m$. Therefore, $\lambda_i=\|\wedge^2 h_*\|^{\frac{1}{2}}$ for $1\leq i\leq 2m$.\par
   \textit{Odd Dimensional Case $n=2m-1\geq3$}:
Let $M\times \mathbb{S}_r^1$ be the Riemannian product of $M$ with the circle $\mathbb{S}_r^1$ of radius $r\geq1$. Consider the following map
\[(M\times \mathbb{S}_r^1,g+r^2ds^2)\xrightarrow{h\times\frac{\mathrm{id}}{r}}(\mathbb{S}^{2m-1}\times\mathbb{S}^1,g_{\mathbb{S}^{2m-1}}+ds^2)\xrightarrow{\,\gamma\,}\mathbb{S}^{2m-1}\wedge\mathbb{S}^{1}\cong\mathbb{S}^{2m},\]
where $h\times\frac{\mathrm{id}}{r}$ is given by $(h\times\frac{\mathrm{id}}{r})(x,t)=(h(x),\frac{t}{r})$ for all $(x,t)\in M\times \mathbb{S}_r^1$, and $\gamma$ is a distance non-increasing map of nonzero degree.\par
Define $\tilde{h}=\gamma\circ(h\times\frac{\mathrm{id}}{r})$. Consider two complex spinor bundles $S$ over $M\times \mathbb{S}_r^1$ and $E_0$ over $\mathbb{S}^{2m}$, respectively. And consider the bundle $S\otimes E$ over $M\times \mathbb{S}_r^1$ for $E=\tilde{h}^*E_0$.\par
Let $\{\tilde{e}_i\}_{i=1}^{2m}$ be a $g+r^2ds^2$-orthonormal frame near $(p,\theta) \in M\times \mathbb{S}_r^1$ such that $\tilde{e}_1,...,\tilde{e}_{2m-1}$ are tangent to $M$ and $\tilde{e}_{2m}$ is tangent to $\mathbb{S}_r^1$. Let $\{\tilde{\epsilon}_i\}_{i=1}^{2m}$ be a $g_{\mathbb{S}^{2m}}$-orthonormal frame near $\tilde{h}(p,\theta) \in \mathbb{S}^{2m}$. Moreover, we can require that there exist non-negative scalars $\{\tilde{\lambda}_i\}_{i=1}^{2m}$ such that $\tilde{h}_*(\tilde e_i)=\tilde{\lambda}_i\tilde\epsilon_i$. Note that 
\[\tilde{\lambda}_i^2=g_{\mathbb{S}^{2m}}(\tilde{h}_*(\tilde e_{i}),\tilde{h}_*(\tilde e_{i}))\leq g_{\mathbb{S}^{2m-1}}(h_*(\tilde e_{i}),h_*(\tilde e_{i}))\leq\|h_*\|^2\leq K^2,\]
\[\tilde\lambda_i^2\tilde\lambda_j^2=g_{\mathbb{S}^{2m}}(\tilde{h}_*(\tilde e_{i}\wedge\tilde e_j),\tilde{h}_*(\tilde e_{i}\wedge\tilde e_j))\leq g_{\mathbb{S}^{2m-1}}(h_*(\tilde e_{i}\wedge\tilde e_j),h_*(\tilde e_{i}\wedge\tilde e_j))\leq\|\wedge^2 h_*\|^2\]
and
\[\tilde\lambda_{2m}^2=g_{\mathbb{S}^{2m}}(\tilde{h}_*(\tilde e_{2m}),\tilde{h}_*(\tilde e_{2m}))\leq ds^2(\frac{\tilde e_{2m}}{r},\frac{\tilde e_{2m}}{r})=\frac{1}{r^4}\leq\frac{1}{r^2},\]
where $K$ is a positive constant independent of $r$ and $1\leq i\neq j\leq2m-1$. Thus, $0\leq\tilde\lambda_i\leq\|h_*\|\leq K$, $0\leq\tilde\lambda_i\tilde\lambda_j\leq\|\wedge^2 h_*\|$ and $0\leq\tilde\lambda_{2m}\leq\frac{1}{r}$.\par
As before, according to Lemma 4.5 in \cite{LlarullMR1600027}, we have 
\begin{equation*}
\begin{split}
\langle\mathcal{R}^E\Phi,\Phi\rangle=&\frac{1}{4}\sum_{i\neq j=1}^{2m-1}\sum_{k,l}\sum_{\alpha,\beta}\tilde\lambda_i\tilde\lambda_ja_{\alpha\beta}a_{kl}\langle e_i\cdot e_j\cdot\sigma_{\alpha}\otimes\epsilon_j\cdot\epsilon_i\cdot v_\beta,\sigma_{k}\otimes v_l\rangle\\
&+\frac{1}{4}\sum_{i=1}^{2m-1}\sum_{k,l}\sum_{\alpha,\beta}\tilde\lambda_i\tilde\lambda_{2m}a_{\alpha\beta}a_{kl}\langle e_i\cdot e_{2m}\cdot\sigma_{\alpha}\otimes\epsilon_{2m}\cdot\epsilon_i\cdot v_\beta,\sigma_{k}\otimes v_l\rangle\\
&+\frac{1}{4}\sum_{j=1}^{2m-1}\sum_{k,l}\sum_{\alpha,\beta}\tilde\lambda_{2m}\tilde\lambda_{j}a_{\alpha\beta}a_{kl}\langle e_{2m}\cdot e_{j}\cdot\sigma_{\alpha}\otimes\epsilon_{j}\cdot\epsilon_{2m}\cdot v_\beta,\sigma_{k}\otimes v_l\rangle\\
&\geq-\frac{1}{4}\sum_{i\neq j=1}^{2m-1}\tilde\lambda_i\tilde\lambda_j\|\Phi\|^2-\frac{1}{2}\sum_{i=1}^{2m-1}\tilde\lambda_i\tilde\lambda_{2m}\|\Phi\|^2\\
&\geq-\frac{1}{4}\sum_{i\neq j=1}^{2m-1}\tilde\lambda_i\tilde\lambda_j\|\Phi\|^2-\frac{(2m-1)K}{2r}\cdot\|\Phi\|^2.
\end{split}
\end{equation*}
Consequently,
\begin{equation}\label{eq-key}
\begin{split}
\int_{M\times\mathbb{S}_r^1}\langle D_{f,E}^2\Phi,\Phi\rangle e^{-f}\geq&\int_{M\times \mathbb{S}_r^1}\| \nabla\Phi\|^2 e^{-f}+\frac{1}{4}\int_{M\times \mathbb{S}_r^1}R_{f}\|\Phi\|^2e^{-f}\\&-\frac{1}{4}\int_{M\times \mathbb{S}_r^1}\left(\sum_{i\neq j=1}^{2m-1}\tilde\lambda_i\tilde\lambda_j+\frac{2(2m-1)K}{r}\right)\|\Phi\|^2e^{-f}.
  \end{split}
  \end{equation}
\par
Next, we follow Wang-Xie \cite{W-XMR4853305} to prove $R_{f}=(2m-1)(2m-2)\|\wedge^2 h_*\|$ by contradiction. Assume to the contrary that the inequality $R_{f}\geq(2m-1)(2m-2)\|\wedge^2 h_*\|$ is strict somewhere. Then there exists $x_0\in M$ and two positive constants $\delta,\delta_0>0$ such that for any $y\in U_{\delta}(x_0)$ with
\[\left(R_{f}-(2m-1)(2m-2)\|\wedge^2 h_*\|\right)(y)>\delta_0,\]
where $U_{\delta}(x_0)$ is a $\delta$-neighborhood of $x_0$ in $M$.\par
By Atiyah-Singer index theorem and index is homotopy invariant, we know that $0\neq Index(D^+_{E^+})=Index(D^+_{f,E^+})$ (it is equal to Euler characteristic of $\mathbb{S}^{2m}$ times degree of $\tilde{h}$). Therefore, $ker(D_{f,E})\neq0$, there exists $0\neq\Psi\in\Gamma(S\otimes E)$ such that $D_{f,E}\Psi=0$. By Lemma 2.8 in \cite{W-XMR4853305}, we have
\begin{equation}\label{Lem-W}
\begin{aligned}
\int_{M\times \mathbb{S}_r^1}\|\Psi\|^2e^{-f}\leq C\int_{U_{\delta}(x_0)\times \mathbb{S}_r^1}\|\Psi\|^2e^{-f}+C\int_{M\times \mathbb{S}_r^1}\|\nabla\Psi\|^2e^{-f},
\end{aligned}
  \end{equation}
here we take the metric $e^{-f}\cdot\langle\,,\rangle$ on spinor bundle $S\otimes E$ and $C$ is a positive constant which is independent of $r$.\par
Since $D_{f,E}\Psi=0$, combining inequality \eqref{eq-key} with inequality \eqref{Lem-W},  it follows that 
\begin{equation*}
\begin{split}
0\geq&\int_{M\times \mathbb{S}_r^1}\| \nabla\Psi\|^2 e^{-f}+\frac{1}{4}\int_{M\times \mathbb{S}_r^1}\left(R_f-\sum_{i\neq j=1}^{2m-1}\tilde\lambda_i\tilde\lambda_j\right)\|\Psi\|^2e^{-f}\\&-\frac{2m-1}{2r}\int_{M\times \mathbb{S}_r^1}\|{h}_*\|\|\Psi\|^2e^{-f}\\
\geq&\int_{M\times \mathbb{S}_r^1}\| \nabla\Psi\|^2 e^{-f}+\frac{\delta_0}{4}\int_{U_{\delta}(x_0)\times \mathbb{S}_r^1}\|\Psi\|^2e^{-f}\\&-\frac{(2m-1)K}{2r}\int_{M\times \mathbb{S}_r^1}\|\Psi\|^2e^{-f}\\
\geq&\left(1-\frac{(2m-1)KC}{2r}\right)\int_{M\times \mathbb{S}_r^1}\| \nabla\Psi\|^2 e^{-f}\\&+\left(\frac{\delta_0}{4}-\frac{(2m-1)KC}{2r}\right)\int_{U_{\delta}(x_0)\times \mathbb{S}_r^1}\|\Psi\|^2e^{-f}.
\end{split}
\end{equation*}
Since $C$ and $K$ are independent of $r$, for a given sufficiently large $r$, the above estimates imply that $\Psi$ vanishes on $U_{\delta}(x_0)\times \mathbb{S}_r^1$ and $\nabla\Psi$ vanishes on $M\times \mathbb{S}_r^1$. This together with the inequality \eqref{Lem-W} implies that $\Psi$ vanishes on $M\times \mathbb{S}_r^1$, which leads to a contradiction. Hence, we have proved that $R_{f}=(2m-1)(2m-2)\|\wedge^2 h_*\|$. By the same argument, we also have 
\[\sum_{i\neq j=1}^{2m-1}\tilde\lambda_i\tilde\lambda_j=(2m-1)(2m-2)\|\wedge^2 h_*\|,\]
combining with $\tilde\lambda_{i}\tilde\lambda_{j}\leq\|\wedge^2 h_*\|$ for $1\leq i\neq j\leq2m-1$, we see that $\tilde\lambda_{i}\tilde\lambda_{j}=\|\wedge^2 h_*\|$ for $1\leq i\neq j\leq2m-1$. Suppose for $i<j$ that $\tilde e_i\wedge\tilde e_j=\sum_{\bar{i}<\bar{j}}T_{ij}^{\bar{i}\bar{j}}e_{\bar{i}}\wedge e_{\bar{j}}$ near $p\in M$, note that $\sum_{\bar{i}<\bar{j}}\left(T_{ij}^{\bar{i}\bar{j}}\right)^2=1$ since $\tilde e_i\wedge\tilde e_j$ is a unit 2-vector. We have

\begin{align*}
\|\wedge^2 h_*\|^2=\tilde\lambda_{i}^2\tilde\lambda_{j}^2&=g_{\mathbb{S}^{2m}}(\tilde{h}_*(\tilde e_{i}\wedge\tilde e_j),\tilde{h}_*(\tilde e_{i}\wedge\tilde e_j))\\
&\leq g_{\mathbb{S}^{2m-1}}(h_*(\tilde e_{i}\wedge\tilde e_j),h_*(\tilde e_{i}\wedge\tilde e_j))\\
&=\sum_{\bar{i}<\bar{j}}\sum_{\bar{k}<\bar{l}}T_{ij}^{\bar{i}\bar{j}}T_{ij}^{\bar{k}\bar{l}}g_{\mathbb{S}^{2m-1}}(h_*( e_{\bar{i}}\wedge e_{\bar{j}}),h_*(\tilde e_{\bar{k}}\wedge\tilde e_{\bar{l}}))\\
&=\sum_{\bar{i}<\bar{j}}\sum_{\bar{k}<\bar{l}}T_{ij}^{\bar{i}\bar{j}}T_{ij}^{\bar{k}\bar{l}}\lambda_{\bar{i}}\lambda_{\bar{j}}\lambda_{\bar{k}}\lambda_{\bar{l}}\delta_{\bar{i}\bar{k}}\delta_{\bar{j}\bar{l}}\\
&=\sum_{\bar{i}<\bar{j}}\left(T_{ij}^{\bar{i}\bar{j}}\right)^2\lambda_{\bar{i}}^2\lambda_{\bar{j}}^2\\
&\leq\|\wedge^2 h_*\|^2.\\
\end{align*}
Thus, $\lambda_{i}\lambda_{j}=\|\wedge^2 h_*\|$ for $1\leq i\neq j\leq 2m-1$, which implies that $\lambda_{i}=\|\wedge^2 h_*\|^{\frac{1}{2}}$ for $1\leq i\leq 2m-1$.\par

So far, we have proved that $R_f=n(n-1)\|\wedge^2 h_*\|$ and $\lambda_i=\|\wedge^2 h_*\|^{\frac{1}{2}}$ for $1\leq i\leq n$. Define $u\coloneqq\|\wedge^2 h_*\|=\lambda_i^2$ and $U=\{x\in M:u(x)\neq0\}$, note that $u$ is a smooth function on $M$. Since $h$ has non-zero degree, $U$ is non-empty. Thus, $h^*g_{\mathbb{S}^n}=u\cdot g$ on $U$. By the formula of scalar curvature under conformal change, we have on $U$
\begin{equation}
n(n-1)=h^*[R(g_{\mathbb{S}^n})]=u^{-1}R(g)-(n-1)u^{-2}\Delta u-\frac{(n-1)(n-6)}{4}u^{-3}|\nabla u|^2.
\end{equation}
Since $R_f(g)=R(g)+2\Delta f-|\nabla f|^2=n(n-1)u$, $u$ and $\nabla u$ vanish on $M\setminus U$, we can deduce that 
\begin{equation}\label{main-eq}
u^ke^{lf}\left(-2u^{2}\Delta f+u^{2}|\nabla f|^2-(n-1)u\Delta u-\frac{1}{4}(n-1)(n-6)|\nabla u|^2\right)=0
\end{equation}
on the entire $M$ for all $k\geq0$ and $l\in R$. By integration by parts, we have 
\begin{equation}\label{eq-ibp}
    \begin{split}
0=&(1+2l)\int_{M}u^{k+2}e^{lf}|\nabla f|^2\\
&+(n-1)\left(k+1-\frac{n-6}{4}\right)\int_{M}u^ke^{lf}|\nabla u|^2\\
&+(2k+4+ln-l)\int_{M}u^{k+1}e^{lf}g(\nabla u,\nabla f).
    \end{split}
\end{equation}
By taking 
\begin{align}
l=\frac{-n^2+(11+4\sqrt{2})n-10-4\sqrt2}{2(n-1)^2}~~~ \text{and}~~
k=\frac{n-2}{4}\geq0,
\end{align}
we have
\begin{align*}
1+2l=&\left(\frac{1+2\sqrt{2}}{\sqrt{n-1}}\right)^2,\\
2k+4+ln-l=&2(4+\sqrt{2}),\\
(n-1)\left(k+1-\frac{n-6}{4}\right)=&2(n-1).
\end{align*}
It follows from equation \eqref{eq-ibp} that 
\begin{equation*}
    \int_{M}u^{\frac{n-2}{4}}e^{lf}|Au\nabla f+B\nabla u|^2=0,
\end{equation*}
where $A=\frac{1+2\sqrt{2}}{\sqrt{n-1}}$, $B=\sqrt{2(n-1)}$. Thus, $u^{\frac{n-2}{4}}|Au\nabla f+B\nabla u|^2=0$ on $M$. Moreover, we have
\begin{equation}\label{eq-equiv}
    Au\nabla f+B\nabla u\equiv0~~~\text{on}\,\,M.
\end{equation}
Consequently,
\begin{equation}\label{eq-div}
\begin{split}
0=&u\times\mathrm{div}\left(Au\nabla f+B\nabla u\right)\\=&Au^2\Delta f+Aug(\nabla u,\nabla f)+Bu\Delta u\\
=&Au^2\Delta f-B|\nabla u|^2+Bu\Delta u.\\
\end{split}
\end{equation}
Plugging $u\nabla f=-\frac{B}{A}\nabla u$ and equation \eqref{eq-div}  into equation \eqref{main-eq} for $k=0,\,m=0$, we obtain
\[\left(2\frac{B}{A}-(n-1)\right)u\Delta u+\left(\frac{B^2}{A^2}-2\frac{B}{A}-\frac{1}{4}(n-1)(n-6)\right)|\nabla u|^2=0,\]
by integration by parts, we see that
\[\frac{1}{196}(n-1)\left((23-32\sqrt{2})n-30+144\sqrt{2}\right)\int_{M}|\nabla u|^2=0.\]
Therefore, $\nabla u\equiv0.$ Since $h$ has non-zero degree, $u$ is a positive constant on M, i.e. $u\equiv c>0.$ It follows that $h:(M,c\cdot g)\to(\mathbb{S}^n,g_{\mathbb{S}^n})$ is a local isometry. Since $M$ is closed and $\mathbb{S}^n$ is simply connected, we conclude that $h:(M,c\cdot g) \to (\mathbb{S}^n, g_{\mathbb{S}^n})$ is an isometry. $f$ is a constant comes from equation \eqref{eq-equiv}. This completes the proof.
\end{proof}
\section{The Llarull type theorem in the form of $\mathbb{S}^k\texorpdfstring{\times}{x}\mathbb{T}^{n-k}$}\label{sec-HSS}
In this section, we follow the approach of Hao-Shi-Sun \cite{HSSMR4855340} to establish the rigidity result of $\mathbb{S}^k\times\mathbb{T}^{n-k}$. We prove Theorem \ref{HSS}, specifically addressing the case $\mathbb{S}^{n-1}\times\mathbb{T}^{1}$. For the case $\mathbb{S}^{k}\times\mathbb{T}^{n-k}$, $2\leq n-k\leq n-3$, we use the torical symmetrization technique to achieve dimension reduction as described in \cite[Remark 3.3]{HSSMR4855340}(see also \cite[Proposition 2.2]{ZJMR4108854}).\par
Since Theorem \ref{T-extension} is crucial for the proof of Theorem \ref{HSS}, we first prove it.
\begin{proof}[Proof of Theorem \ref{T-extension}]
    By direct calculation, we have
    \begin{equation}\label{eqHSS}
        R(\bar{g})=R(g)+2\Delta f-|\nabla f|^2-\sum_{i=1}^{k}|\nabla f_i|^2,
    \end{equation}
    where $f=\sum_{i=1}^kf_i$.
Since $R(\bar{g})\geq n(n-1)\|\wedge^2 h_*\|$, $R(g)+2\Delta f-|\nabla f|^2\geq n(n-1)\|\wedge^2 h_*\|$. By Theorem \ref{main-theorem-2}, we obtain there exists a constant $c>0$ such that $h:(M,c\cdot g) \to (\mathbb{S}^n, g_{\mathbb{S}^n})$ is an isometry and $f$ is a constant. Note $\|\wedge^2 h_*\|=c$, by equation \eqref{eqHSS}, $f_i$ is a constant for each $1\leq i\leq k$.
\end{proof}
Recall Theorem \ref{HSS} for the specific case $\mathbb{S}^{n-1}\times\mathbb{S}^{1}$.
\begin{theorem}
    Suppose $4\leq n\leq7$ . Let $(M^n, g)$ be a closed spin Riemannian manifold. If $h:(M,g) \to (\mathbb{S}^{n-1}\times\mathbb{S}^{1}, g_{\mathbb{S}^{n-1}}+d\theta^2)$ is a smooth map of non-zero degree such that $R(g)\geq (n-1)(n-2)\|\wedge^2 h_*\|$, then there exists a constant $c>0$ such that $(M,g)$ is locally isometric to $(\mathbb{S}^{n-1}\times\mathbb{S}^{1}, \frac{1}{c}\cdot(g_{\mathbb{S}^{n-1}}+d\theta^2))$.
\end{theorem}
\begin{proof}
    Since $h$ has non-zero degree, according to Sard’s theorem and the geometric measure theory, there exist a $\theta_0\in \mathbb{S}^{1}$ such that $0\neq[h^{-1}(\mathbb{S}^{n-1}\times\{\theta_0\})] \in H_{n-1}(M,\mathbb{Z})$ and a closed 2-sided area-minimizing hypersurface $\Sigma^{n-1} \in [h^{-1}(\mathbb{S}^{n-1}\times\{\theta_0\})]$, and $h_1\coloneqq P_1\circ h|_{\Sigma}: \Sigma\to\mathbb{S}^{n-1}$ has non-zero degree, where $P_1:\mathbb{S}^{n-1}\times\mathbb{S}^{1}\to \mathbb{S}^{n-1}$ is the canonical projection. By the second variation of stable minimal hypersurface, the associated Jacobi operator  $$J_{\Sigma}\coloneqq-\Delta_{\Sigma}-(\mathrm{Ric}_g(\nu,\nu)+|A_\Sigma|^2)=-\Delta_{\Sigma}+\frac{1}{2}(R_{\Sigma}-R(g)-|A_\Sigma|^2)$$is non-negative, where $A_{\Sigma}$ is the second fundamental form of $\Sigma$ and $\nu$ is the unit normal vector of $\Sigma$. Let $\phi$ be the first eigenfunction such that 
    \begin{equation}\label{stability}
        \begin{split}
            J_{\Sigma}\phi=&-\Delta_{\Sigma}\phi-(\mathrm{Ric}_g(\nu,\nu)+|A_\Sigma|^2)\phi\\=&-\Delta_{\Sigma}\phi+\frac{1}{2}(R_{\Sigma}-R(g)-|A_\Sigma|^2)\phi\\
            =&\lambda\phi,
        \end{split}
    \end{equation}
    where $\lambda\geq0$ is the first eigenvalue. By the non-vanishing Theorem in \cite[page 47]{Gromov2}, $\phi>0$. Since $\Sigma$ has trivial normal bundle, it is also a spin manifold. Now, we consider the manifold $(\Sigma\times\mathbb{S}^{1}, g_{\phi}=g|_\Sigma+\phi^2d\theta^2)$, which is spin, by direct calculation, we have on $\Sigma$
    \begin{align*}
    R(g_\phi)=&R_{\Sigma}-\frac{2\Delta_{\Sigma}\phi}{\phi}\\=&R(g)+|A_\Sigma|^2+2\lambda\\
    \geq&(n-1)(n-2)\|\wedge^2 (h|_{\Sigma})_*\|\\
    \geq&(n-1)(n-2)\|\wedge^2 h_{1*}\|.
    \end{align*}
    Since $h_1: \Sigma\to\mathbb{S}^{n-1}$ has non-zero degree, by Theorem \ref{T-extension}, we deduce that $(\Sigma,g|_\Sigma)$ is isometric to $(\mathbb{S}^{n-1},\frac{1}{c}\cdot g_{\mathbb{S}^{n-1}})$ for a constant $c>0$, $\|\wedge^2 h_{1*}\|=c$ and $\phi$ is a constant. Utilizing equation \eqref{stability}, we obtain $\lambda=0$, $A_{\Sigma}=0$ and $\mathrm{Ric}_g(\nu,\nu)=0$. According to Proposition 3.2 in \cite{B2MR2765731}(see also Lemma 3.3 in \cite{ZJMR4108854}), we can construct a local foliation $\{\Sigma_{t}\}_{-\epsilon \leq t\leq \epsilon}$ in $M$ such that $\Sigma_{t}$ are of constant mean curvature and $\Sigma_{0}=\Sigma$, let $H_t$ denote the mean curvature of $\Sigma_t$. We claim that $H_t\equiv0$ for any $t\in(-\epsilon,\epsilon)$. It suffices to prove this for nonnegative $t$, as the argument for the negative side is analogous. For contradiction, suppose that the assertion is not true, then there exists a $t_0>0$ such that $H_{t_0}=\epsilon_0>0$, since otherwise $H_{t}\leq0$ and the area-minimizing property of $\Sigma$ will imply $H_{t}\equiv0$.  Let $\Omega_{t_{0}}$ denote the region enclosed by $\Sigma$ and $\Sigma_{t_{0}}$ i.e.,
		\[\Omega_{t_{0}}:=\cup_{t\in [0,t_{0}]}\Sigma_{t}.\]
		Consider the following functional 
		\begin{equation}\label{eqn-function-h}
			E (\Omega) = \int_{\partial^{\ast} \Omega\setminus\Sigma}\mathrm{d}\mathcal{H}^{n-1} - \int_{\Omega_{t_0}} \epsilon_1(\chi_{\Omega}- \chi_{\bar{\Omega}})\mathrm{d} \mathcal{H}^{n},\,0<\epsilon_1<\epsilon_0,
		\end{equation}
		for $\Omega,\bar{\Omega}\in \mathcal{C}$ and $\Sigma\subset \bar{\Omega}$, where 
		\[\mathcal{C}=\{\Omega|\mbox{ all Caccioppoli sets } \Omega \subset \Omega_{t_{0}} \mbox{ and }\Omega\triangle \bar{\Omega}\Subset \overset{\circ}{\Omega}_{t_{0}}\}. \]
 Due to $0<\epsilon_1<\epsilon_0$, we can find a Borel set $\hat{\Omega}$ minimizing the brane functional $E$ such that $\hat{\Sigma}=\partial \hat{\Omega}\setminus\Sigma$ is a smooth 2-sided hypersurface disjoint from $\Sigma$ and $\Sigma_{t_{0}}$. Let $\hat{\nu}$ and $\hat{A}$ denote the outwards pointing unit normal vector and the second fundamental form of $\hat{\Sigma}$, respectively. Since $\hat{\Sigma}$ is $E$-minimizing, it has constant mean curvature $\epsilon_1$ and it is $E$-stable. So the Jacobi operator  
 $$\hat{J}\coloneqq-\Delta_{\hat g}-(\mathrm{Ric}_g(\hat\nu,\hat\nu)+|\hat A|^2)=-\Delta_{\hat g}+\frac{1}{2}(R_{\hat\Sigma}-R(g)-|\hat A|^2-\epsilon_1^2)$$is non-negative, where $\hat g$ is the induced metric of $\hat\Sigma$ from $(M,g)$. Taking the first eigenfunction $\hat\phi$ of $\hat J$ with respect to the first eigenvalue $\hat\lambda\geq0$ and defining the metric $\tilde g=\hat g+\hat\phi^2d\theta^2$ on $\hat\Sigma\times\mathbb{S}^1$, then 
 \[R(\tilde g)=R(g)+|\hat A|^2+2\hat\lambda+\epsilon_1^2\geq(n-1)(n-2)\|\wedge^2 \hat h_{1*}\|+\epsilon_1^2,\]
 where $\hat h_1\coloneqq P_1\circ h|_{\hat\Sigma}: \hat\Sigma\to\mathbb{S}^{n-1}$, which has non-zero degree since $\hat\Sigma$ is homologous to $\Sigma$. By Theorem \ref{T-extension}, we have $|\hat A|=0$ and $\epsilon_1=0$, which leads to a contradiction.\par
 Since $H_t$ vanishes for any $t\in(-\epsilon,\epsilon),$ there holds $\operatorname{Area}(\Sigma_t)=\operatorname{Area}(\Sigma)$, which yields
	that $\Sigma_t$ is also area-minimizing. Replacing $\Sigma$ by $\Sigma_t$ and combining with Theorem \ref{T-extension}, $\Sigma_{t}$ is totally geodesic and has vanished normal Ricci curvature.
	
	In the following, we establish the existence of a local isometry $\varPhi:\Sigma\times\mathbb{R}\to M$. Let $\tilde{\nu_t}$ and $\tilde{V_t}$  denote the outwards pointing unit normal vector field and normal variation vector field of the foliation $\{\Sigma_{t}\}_{-\epsilon\leq t\leq\epsilon}$, respectively. Assume $\varPhi:\Sigma\times(-\epsilon,\epsilon)\to M$ is the flow generated by $\tilde{V_t}$. It follows that $\varPhi$ is an embedding in a small neighborhood of $\Sigma$ and the pull-back of the metric $g$ is
	\[\varPhi^{\ast}(g)=\psi^{2}dt^{2}+\varPhi^{\ast}_t(g_t),\]
	where $\psi=\langle\tilde{V_t},\tilde{\nu_t}\rangle>0$ is the lapse function (see definition in (3.3) in \cite{B2MR2765731}) and $g_t$ is the induced	metric on $\Sigma_{t}$ from $g$. Since $\Sigma_{t}$ is totally geodesic, we have $\partial_t\varPhi^{\ast}_t(g_t)=2\psi  A_{\Sigma_t}=0$, which implies that $\varPhi^{\ast}_t(g_t)=g|_{\Sigma}$. Additionally, the stability condition yields $-\Delta_{\Sigma_t} \psi - |A_{\Sigma_{t}}|^2 \psi
	-\operatorname{Ric}(\nu_t,\nu_t)\psi=-\Delta_{\Sigma_t} \psi=0,$ then $\psi(\cdot,t)\equiv\mathrm{constant}$. Let
	\[r(\cdot,t)=\int_{0}^{t}\psi(\cdot,s)ds,\]
	then \[\varPhi^{\ast}(g)=dr^2+g|_{\Sigma}.\]
	Hence, $\varPhi: \Sigma\times(-\epsilon,\epsilon)\to M$ is a local isometry. Through a continuous argument as Proposition 3.8 in \cite{B2MR2765731}, we conclude that there exists a local isometry $\varPhi: \Sigma\times\mathbb{R}\to M$. Since $(\Sigma,g|_\Sigma)$ is isometric to $(\mathbb{S}^{n-1},\frac{1}{c}\cdot g_{\mathbb{S}^{n-1}})$ for a constant $c>0$, by scaling the fact $\mathbb{R}$, $(M,g)$ is isometrically covered by $(\mathbb{S}^{n-1}\times\mathbb{R}, \frac{1}{c}\cdot(g_{\mathbb{S}^{n-1}}+d\theta^2))$.
\end{proof}
\bibliographystyle{amsplain}
	\bibliography{mybib.bib}
\end{document}